\documentclass[12pt]{article}
\usepackage{amssymb,latexsym,amsfonts,color,amsmath,epsfig,amsthm, cite,bbm,mathrsfs, graphicx} 
\usepackage{indentfirst}

\newtheorem{thm}{Theorem}[section]

\newtheorem{definition}[thm]{Definition}
\newtheorem{lem}[thm]{Lemma}

\begin{document}

\begin{center}
\uppercase{\bf Finite Representability of Integers as $2$-Sums}

\bigskip

{\bf Anant Godbole}\\
{\it Department of Mathematics and Statistics,\\ 
East Tennessee State University,\\ Johnson City, TN 37614, USA}\\
{\tt godbolea@etsu.edu}\\ 
\vskip 10pt
{\bf Zach Higgins}\\
{\it Department of Mathematics,\\ University of California (San Diego),\\ La Jolla, CA 92093, USA}\\
\vskip10pt
{\bf Zoe Koch}\\
{\it Department of Mathematics,\\ University of Vermont,\\ Burlington, VT 05405, USA}\\
\vskip 10pt

\end{center}

\def\qed{\vbox{\hrule\hbox{\vrule\kern3pt\vbox{\kern6pt}\kern3pt\vrule}\hrule}}
\def\ms{\medskip}
\def\n{\noindent}
\def\ep{\varepsilon}
\def\G{\Gamma}
\def\lr{\left(}
\def\ls{\left[}
\def\rs{\right]}
\def\lf{\lfloor}
\def\rf{\rfloor}
\def\lg{{\rm lg}}
\def\lc{\left\{}
\def\rc{\right\}}
\def\rr{\right)}
\def\ph{\varphi}
\def\p{\mathbb P}
\def\P{\rm Po}
\def\nk{n \choose k}
\def\ca{{\cal A}}
\def\s{\cal S}
\def\e{\mathbb E}
\def\tv{{d_{\rm TV}}}
\def\cl{{\cal L}}
\def\a{\alpha}
\def\l{\lambda}
\begin{abstract}
A set $\ca=\ca_{k,n}\subset[n]\cup\{0\}$ is said to be an additive $h$-basis if each element in $\{0,1,\ldots,hn\}$ can be written as an $h$-sum of elements of $\ca$ in {\it at least} one way.  We seek multiple representations as $h$-sums,  and, in this paper we make a start by restricting ourselves to $h=2$.  We say that $\ca$ is said to be a truncated $(\alpha,2,g)$ additive basis if 
each $j\in[\alpha n, (2-\alpha)n]$ can be represented as a $2$-sum of elements of $\ca_{k,n}$ in at least $g$ ways.  In this  paper, we provide sharp asymptotics for the event that a randomly selected set is a truncated $(\alpha,2,g)$ additive basis with high or low probability.
\end{abstract}


\section{Introduction}   
\subsection{Balls in Boxes}We start by introducing results from the classical theory of the random allocation of balls to boxes.  We will be seeing, in the rest of the paper, how and to what extent the results apply to situations such as coverage of integers by $h$-sets of integers.  

Suppose that we are trying to ``pack" balls in boxes so that each box contains at most one ball.  This is the so-called ``birthday problem", and it is well-known, e.g.,  \cite{bh} that if we randomly throw $n$ balls into $N$ boxes, then the threshold for the property to hold with high or low probability (whp or wlp) is $n=\sqrt{N}$, in the sense that if $n/\sqrt{N}\to\infty$, then the probability that each box contains at most one ball is asymptotically 0, and this probability is asymptotically 1 if $n/\sqrt{N}\to0$.  Here and throughtout this paper, we will describe these two situations by using the notation $n\gg\sqrt{N}$ and $n\ll\sqrt{N}$ respectively.  There is a generalization of the birthday threshold to to ``at most $g$ balls, which we rederived in \cite{gghk} using Talagrand's inequality \cite{as}:
\begin{thm}  When $n$ balls are randomly and uniformly distributed in $N$ boxes, then letting $X=X_g$ denote the number of boxes with $g+1$ or more balls,

$$n\ll N^{g/(g+1)}\Rightarrow \p(X=0)\to1,$$
and
$$n\gg N^{g/(g+1)}\Rightarrow \p(X=0)\to0.$$
\end{thm}

Theorem 1.1 exhibits a progression of thresholds, which get close to $n=N$ as $g\to\infty$.
  It may still be the case, however, that not all boxes will have a ball in them if $n\gg N$, which leads us to the question of the coverage of each box by at least one ball, or ``coupon collection".  It is well known that the expected waiting time $\e(W)$ for each of the boxes to be filled is $N(\ln N+\gamma+o(1)),$ where $\gamma$ is Euler's constant, and that the variance of the waiting time is $\sim N^2$.  Together with these facts, Chebychev's inequality can be use to prove, with $X$ denoting the number of empty boxes, and $\omega(1)$ an arbitrary function tending to $\infty$, that
\begin{thm}
$$n=N(\ln N+\omega(1))\Rightarrow \p(W\ge n)=\p(X\ge1)\to0,$$
and
$$n=N(\ln N-\omega(1))\Rightarrow \p(W\le n)=\p(X=0)\to0.$$
\end{thm}
Various people have asked about covering each box $g$ or more times.
Generalizing work of Erd\H os and R\'enyi; and Newman and Shepp; Holst produced the following definitive result:
\begin{thm}  Letting $V_g$ denote the waiting time until each box has at least $g$ balls, we have $$\e(V_g)=N(\ln N+(g-1)
\ln \ln N+\gamma-\ln (g-1)!+o(1)).$$Normalizing by setting
$V^*_g=V_g/N-\ln N-(g-1)\ln\ln N+\ln(g-1)!+o(1)$, we have that
$V_2,\ldots, V_g$ are asymptotically independent.  Moreover $$\p(V^*_g\le u)\to\exp\{-e^{-u}\}.$$  
\end{thm}
From Theorem 1.3, it is easy to derive the following result
\begin{thm}
$$n=N(\ln N+(g-1)\ln\ln N+\omega(1))\Rightarrow \p(X_g=0)\to1,$$
and
$$n=N(\ln N+(g-1)\ln\ln N-\omega(1))\Rightarrow \p(X_g=0)\to0,$$
where $X_g$ is the number of boxes with $\le g-1$ balls.
\end{thm}
Of particular note is the linearity (in $\ln\ln N$) for coverings beyond the first, showing that an additional iterated logarithmic fraction suffices for each subsequent covering (which are asymptotically independent!)  We hope to show that many of these features stay intact even as dependence is introduced into the covering scenarios.  As a final note, we observe that extremal behaviour in the ``balls in boxes" example is trivial:  The maximal number of balls that may be placed in $N$ boxes so that each contains at most one box is $N$, as is the smallest number of balls so as to guarantee at least one ball per box.
 \subsection {Dependence}
A set $\ca\subseteq[n]$ is said to be a $B_h$ set (the totality of these for all $h\ge 2$ are known as Sidon sets) if each of the ${{\vert\ca\vert+h-1}\choose {h}}$ sums of $h$ 
elements drawn with replacement from $\ca$ are distinct.  A set $\ca\subseteq[n]\cup\{0\}$ is said to be an $h$-additive basis if each $j\in [n]$ can be written as the sum of $h$ elements in $\ca$.  Thus, a set is $h$-Sidon or an $h$-additive basis if each element in the potential sumset can be obtained in at most one or at least one way using elements of $\ca$.  It is known that maximal Sidon sets and minimal additive bases are both of order $n^{1/h}$; for example minimal 2-additive bases have size $1.463\sqrt{n}\le\vert\ca\vert\le1.871\sqrt{n}$. See \cite{gjlr} and \cite{gllt} for details on such results regarding Sidon sets and additive bases respectively.  
We are interested, however, in random versions of these results, and we start by noting that the corresponding balls in boxes model is as follows:

The balls are the integers randomly chosen from $[n]$.  However they do not ``go into a single box".  Rather, each ball colludes with other chosen balls, including itself, generating sums with multisets of $h-1$ other balls.  A ball is then placed into the box corresponding to each generated sumset.  For example, if $h=3$ and the balls drawn in sequence are 4, 2, and 6, then balls are placed in boxes $$12(=4+4+4),$$$$6(=2+2+2), 8(=2+2+4), 10(=2+4+4),$$ $$18(6+6+6), 12 (6+2+4), 10 (6+2+2), 14(6+4+4), 14(6+6+2), 16(6+6+4)$$ where the numbers in the three lines indicate what occurs with balls 4, 2, and 6 respectively.  There are clearly several layers of dependence in the allocation of balls to boxes.

Three known facts in the area of thresholds for the emergence of Sidon sets and additive bases are stated next:

\begin{thm} (\cite{gjlr}) Consider a subset $A_n$ of random size obtained by choosing each integer in $[n]$ independently with probability $p=p_n=\frac{k_n}{n}$.
Then for any $h\ge2$,
$$k_n=o(n^{1/2h})\Rightarrow\p(A_n\ {\rm is\ }
B_h)\to1\quad(n\to\infty)$$
and
$$n^{1/2h}=o(k_n)\Rightarrow\p(A_n\ {\rm is\ }
B_h)\to0\quad(n\to\infty).$$  
\end{thm}
In \cite{gghk}, we find the following definition that is related to the original question of Sidon (see \cite{o} and also the second open question in Section 3 below).
\begin{definition}
We say that $\ca\subseteq[n]$ satisfies the $B_h[g]$ property for integers $h\ge2; g\ge 1$ if for all integers $k, h\le k\le nh$, $k$ is realized in at most $g$ ways as a sum
$$a_1+a_2+\ldots+a_h=k$$ for $a_1\le a_2\ldots\le a_h$ and $a_i\in\ca$ for each $i$.
\end{definition}
The authors of \cite{gghk} go on to generalize Theorem 1.5 as follows:
\begin{thm}
Let $\ca\subseteq [n]$ be a random subset of $[n]$ in which each element of $[n]$ is selected for membership in $\ca$ independently with probability $p:=\frac{k}{n}$. Then for any $h \geq 2$, $g \geq 1$ we have:

\[
k = o\left ( n^{\frac{g}{h(g+1)}} \right ) \Rightarrow \p(\ca \text{ is } B_h[g]) \to1\quad(n\to\infty),
\]
and
\[ 
n^{\frac{g}{h(g+1)}} = o\left ( k \right ) \Rightarrow \p(\ca \text{ is } B_h[g]) \to0\quad(n\to\infty).\]
\end{thm}
In transitioning to the case of additive bases, we first note, as in \cite{gllt}, that a single input probability for integer selection will cause edge effect issues.  For example, for $h=2$, since the only way to represent 1 as a 2-sum is as $1+0$, both 0 and 1 {\em must} be selected in order for 1 to be represented.  For this reason, we say that $\ca\subseteq[n]\cup\{0\}$ is a truncated $(\a,h)$ additive basis  if each integer in $[\a n,(h-\a)n]$ can be written as an $h$-sum of elements in $\ca$.
\begin{thm} (\cite{gllt}) 
If we choose elements of $\{0\}\cup[n]$ to be in $\ca$ with probability  
$$p=\sqrt[h]{\frac{
K\log n-K \log{\log{n}}+A_n}{n^{h-1}}},$$ where $K=K_{\a,h}=\frac{h!(h-1)!}{\a^{h-1}}$, 
then
$$\p(\ca\ is\ a\ truncated\ (\a,h)\ additive\ basis)\rightarrow\begin{cases} 0 & \mbox{if}\ A_n\rightarrow-\infty\\  1 & \mbox{if}\ A_n\rightarrow \infty \\ \exp\{- \frac{2\alpha}{h-1}e^{-A/K}\} & \mbox{if}\ A_n\rightarrow A\in{\mathbb R}\end{cases}.$$
\end{thm}
Even though edge effects {\it can} be eliminated by considering modular additive bases, here we consider the truncated additive basis case, where the target sumset is reduced via the parameter $\a$ -- since we are using the same probability $p$ of selection.  The case $h=2$ is studied in greater detail in the next result, which addresses coverage of each sum $g$ times, and is which the main result of this paper.
\begin{thm} If we choose elements of $\{0\}\cup[n]$ to be in $\ca$ with probability  
$$p=\sqrt{\frac{\frac{2}{\a}\log n+(g-2)\frac{2}{\a} \log{\log{n}}+A_n}{n}},$$ then
$$\p(\ca\ is\ a\ truncated\ (\a,2,g)\ basis)\rightarrow\begin{cases} 0 & \mbox{if}\ A_n\rightarrow-\infty\\  1 & \mbox{if}\ A_n\rightarrow \infty \\ \exp\lc- {\frac{2\alpha}{(g-1)!}}e^{-A\a/2}\rc & \mbox{if}\ A_n\rightarrow A\in{\mathbb R}\end{cases},$$ where a $(\a,2,g)$ basis is one for which each integer in the target set can be written as a 2-sum in $g$ ways.
\end{thm}

Both Theorems 1.8 and 1.9  are finite representability versions of the key result in \cite{et}, where a variable input probability was used and the focus was on representing each integer as a sum in {\it logarithmically many} ways.  See also the key results in \cite{gglz}, where logarithmic representability is studied in the context of a single input probability.  Theorem 1.9 exhibits the $\log\log$ phenomenon that arose in the context of Coupon Collection.  Interestingly, though, the $\log\log$ factor is present for the first covering with a negative contribution, disappears for the second, and then reappears with a positive sign.  The paper \cite{gghk} provides many more examples of this phenomenon in a variety of covering and packing situations, specifically those that arise in the context of combinatorial designs, permutations, and union free set families.
\section{Proof of Theorem 1.9}  We start by defining the key random variable of interest.  Let $X=X_g$ be the number of integers in $[\a n, (2-\a)n]$ that are represented as a $2$ sum in fewer than $g$ ways, i.e., in $\le g-1$ ways.  The threshold we seek to establish is for $\p(X=0)$, and, as in so many instances where we employ the Poisson paradigm (see, e.g., \cite{as}), this transition occurs at the level at which $\e(X)$ rapidly transitions from asymptotically 0 to asymptotically $\infty$; this is because $\p(X=0)\sim e^{-\e(X)}=e^{-\l}$.  Towards this end we next carefully estimate $\l$.  We have that
\[X=\sum_{j=\a n}^{(2-\a)n}I_j,\]
where $I_j$ is the indicator of the event that the integer $j$ is underrepresented as defined above.  By linearity of expectation,
\begin{eqnarray}\l=\e(X)&=&\sum_{j=\a n}^{(2-\a)n}\p(j\ {\rm is\ underrepresented})\nonumber\\
&\sim&2\sum_{j=\a n}^n\p(j\ {\rm is\ underrepresented})\nonumber\\
&=&2\sum_{j=\a n}^n\sum_{s=0}^{g-1}{{\lfloor j/2\rfloor}\choose{s}}p^{2s}(1-p^2)^{\lfloor j/2\rfloor-s},\end{eqnarray}
where the last equality follows from the fact that each $j$ can be represented as the sum of $\lfloor j/2\rfloor$ {\em disjoint} pairs of integers.   
Trivially, we have
\begin{equation}\l\ge2\sum_{j=\a n}^n{{\lfloor j/2\rfloor}\choose{g-1}}p^{2g-2}(1-p^2)^{\lfloor j/2\rfloor-g+1},\end{equation}
and Theorem A.2.5 (iii) in \cite{bhj}, which estimates the left tail of a binomial random variable with large mean by its last term yields
\begin{eqnarray}\l&\le& 2\sum_{j=\a n}^n\frac{{jp^2}/{2}-gp^2}{jp^2/2+1-g-p^2}{{\lfloor j/2\rfloor}\choose{g-1}}p^{2g-2}(1-p^2)^{\lfloor j/2\rfloor-g+1}(1+o(1))\nonumber\\
&=&2\sum_{j=\a n}^n{{\lfloor j/2\rfloor}\choose{g-1}}p^{2g-2}(1-p^2)^{\lfloor j/2\rfloor-g+1}(1+o(1)).\end{eqnarray} In deriving (3), we need to know that $jp^2\to\infty$ for $j$'s in the selected range.  This is something we can assume, since we are seeking a threshold at $p\sim{\sqrt{K\log n/n}}$, and we can suppose up front, e.g., that $p\ge{\sqrt{\log\log n/n}}$.  Equations (2) and (3) reveal that
\begin{eqnarray}\l&\sim&2\sum_{j=\a n}^n{{\lfloor j/2\rfloor}\choose{g-1}}p^{2g-2}(1-p^2)^{\lfloor j/2\rfloor-g+1}\nonumber\\
&\sim&\frac{2}{(g-1)!}\sum_{j=\a n}^n\lr\frac{jp^2}{2}\rr^{g-1}e^{-jp^2/2}
\end{eqnarray}
where, in the second line of (4) we have used the facts that $p\to0$ and $g$ is finite, and that for $j$'s in the specified range, we have $(1-p^2)^{\lfloor j/2\rfloor}\sim e^{-jp^2/2}$, and ${{\lfloor j/2\rfloor}\choose {g-1}}\sim j^{g-1}/(2^{g-1}(g-1)!)$.  Since the function $x^{g-1}e^{-x}$ is decreasing for $x>g-1$, we see that the summand in (4) will also be decreasing provided, e.g., that $p^2\ge \frac{\log\log n}{\a n}$, which we will assume.  Thus
\begin{eqnarray}
\l&\le&\frac{2}{(g-1)!}\sum_{j=\a n}^\infty \lr\frac{jp^2}{2}\rr^{g-1}e^{-jp^2/2}\nonumber\\
&\le&\frac{4}{p^2(g-1)!}\int_{\a np^2/2}^\infty x^{g-1}e^{-x}dx+o(1)\nonumber\\
&\sim&\frac{4}{p^2(g-1)!}\lr\frac{\a np^2}{2}\rr^{g-1}\exp\{-\a np^2/2\},
\end{eqnarray}
where the last line of (5) follows from the simplest asymptotic estimate (i.e., without error terms) of the incomplete gamma function.  We first check to see when $\l\to 0$.  We start by letting $p^2=2\log n/(\a n)$ and find that the right side of (5) is of order
\[\frac{n}{\log n}(\log n)^{g-1}\frac{1}{n}\not\to0\]  for $g\ge 2$.  Making the adjustment $p^2=(2+\ep)\log n/(\a n)$ reveals that the right side of (5) tends to zero at the rate
\[\frac{n}{\log n}(\log n)^{g-1}\frac{1}{n^{1+(\ep/2)}},\] which we seek to improve.  Accordingly, we set 
\[p^2=\frac{2\log n+B\log\log n}{\a n},\]and find that the right side of (5) is of order
\[\frac{n}{\log n}(\log n)^{g-1}\frac{1}{n(\log n)^{B/2}}.\]
Setting $B=2(g-2)$ yields that the right side of (5) is constant, and the incorporation of an additional $A_n$ term, with $A_n\to\infty$, yields $\e(X)\to0$ and leads to the conclusion that
\[\p(X\ge 1)\le\e(X)\to0.\]  Next, we return to (4) and see that
\begin{eqnarray}
\l&\ge&\frac{4}{p^2(g-1)!}\int_{\a np^2/2}^{np^2/2} x^{g-1}e^{-x}dx\nonumber\\
&=&\frac{4}{p^2(g-1)!}\int_{\a np^2/2}^{\infty} x^{g-1}e^{-x}dx-\frac{4}{p^2(g-1)!}\int_{np^2/2}^{\infty} x^{g-1}e^{-x}dx\nonumber\\
&\sim&\frac{4}{p^2(g-1)!}\lr\lr\frac{\a np^2}{2}\rr^{g-1}\exp\{-\a np^2/2\}-\lr\frac{np^2}{2}\rr^{g-1}\exp\{-np^2/2\}\rr\nonumber\\
&=&\frac{4}{p^2(g-1)!}\lr\frac{\a np^2}{2}\rr^{g-1}\exp\{-\a np^2/2\}\lr1-\lr\frac{1}{\a}\rr^{g-1}e^{-(1-\a)np^2/2}\rr\nonumber\\
&\sim&\frac{4}{p^2(g-1)!}\lr\frac{\a np^2}{2}\rr^{g-1}\exp\{-\a np^2/2\}.
\end{eqnarray}
It follows, as with the analysis concerning the $\l\to0$ case, that $\l\to\infty$ when
\[p=p=\sqrt{\frac{\frac{2}{\a}\log n+(g-2)\frac{2}{\a} \log{\log{n}}+A_n}{n}},\]
with $A_n\to-\infty$.

The next (and critical) phase of the proof is to show that $\p(X=0)\approx e^{-\l}$.  We will exhibit this by using the Stein-Chen method of Poisson approximation \cite{bhj}, which will yield that
\[\tv(\cl(X),\P(\l)):=\sup_{A\subseteq{\mathbb Z}^+}\left\vert\p(X\in A)-\sum_{j\in A}\frac{e^{-\l}\l^j}{j!}\right\vert\to0\]
for a range of $p$'s that encompasses our threshold.  (In the above $\cl(Z)$ denotes the distribution of $Z$, $\P(\l)$ the Poisson distribution with parameter $\l$, and $\tv$ the usual total variation distance.)  Setting $A=\{0\}$ will complete the proof.
 
For each $j$ we seek to define an ensemble of auxiliary variables $J_{ji}; \a n\le i\le (2-\a)n$ that satisfy, for each $j$,
\[\cl(J_{ji}: \a n\le i\le (2-\a)n)=\cl(I_i,  \a n\le i\le (2-\a)n\vert I_j=1).\]
We do this as follows:  If $I_j=1$, we simply set $J_{ji}=I_i$ for each $i$.  If, however, the integer $j$ is represented $g$ or more times, we deselect one or more integer so as to achieve the distribution corresponding to $I_j=1$.  We then set $J_{ji}=1$ if integer $i$ is represented $g-1$ or fewer times after the deselection.  Now the exact nature of this {\em coupling} is not important (and, in fact, is rather complicated), but what is evident is that $J_{ji}\ge I_i$, since $J_{ji}=0; I_i=1$ would entail that $i$ flipped from being underrepresented to being represented $\ge g$ times {\em after} the deselection of some integers.  In other words, the indicator variables are {\em positively related} and Corollary 2.C.4 in \cite{bhj} applies, so that we have 
\begin{eqnarray}
\tv(\cl(X),\P(\l))&\le&\frac{1-e^{-\l}}{\l}\lr Var(X)-\l+2\sum_j\p^2(I_j=1)\rr\nonumber\\
&\le&\frac{1}{\l}\lr\sum_j\p^2(I_j=1)+\sum_j\sum_\ell\lc\e(I_jI_\ell)-\e(I_j)\e(I_\ell)\rc\rr\nonumber\\
&=&T_1+T_2, \quad{\rm say}.\end{eqnarray}
We begin with $T_1$.
\begin{eqnarray}
T_1&\le&\frac{1}{\l}\max_j\p(I_j=1)\sum_j\p(I_j=1)\nonumber\\
&=&\p(I_{\a n}=1)\nonumber\\
&=&\sum_{j=0}^{g-1}{{\lfloor \a n/2\rfloor}\choose {j}}p^{2j}(1-p^2)^{{\lfloor \a n/2\rfloor}-j}\nonumber\\
&\le&{{\lfloor \a n/2\rfloor}\choose {g-1}}p^{2g-2}(1-p^2)^{{\lfloor \a n/2\rfloor}-g+1}(1+o(1))\nonumber\\
&\le&\frac{1}{(g-1)!}\lr\frac{\a np^2}{2}\rr^{g-1}e^{-\a np^2/2}(1+o(1))\nonumber\\
&\to&0
\end{eqnarray}
provided that $np^2\to\infty$, which we may assume without any loss.  Clearly the correlation term $T_2$ will dictate the closeness of the Poisson approximation.  Our first lemma shows that while computing $\p(I_jI_\ell=1)$, it suffices to consider the case where the sumsets for $j$ and $\ell$ are disjoint.
\begin{lem}For some constant $K$, we have that for each $j,\ell$, \[\p(I_jI_\ell=1)\le\p(I_j=1)\p(I_\ell=1)\lr1+\frac{K}{np}\rr.\]
\end{lem}
\begin{proof}We let $A_{j,\ell,r,s}$ denote the event that integers $j, \ell$ are represented $r$ and $s$ times respectively.  Likewise, let $B_{j,\ell,r,s}$ denote the event that integers $j, \ell$ are represented in $r$ and $s$ entirely disjoint ways.  \begin{eqnarray}
\p(I_jI_\ell&=&1)\nonumber\\&=&\sum_{r,s=0}^{g-1}\p(A_{j,\ell,r,s})\nonumber\\
&=&\sum_{r,s=0}^{g-1}\p(B_{j,\ell,r,s})+\sum_{r,s=0}^{g-1}\p(B^C_{j,\ell,r,s})
\end{eqnarray}  We first calculate the contribution to (9) of the disjoint case:
\begin{eqnarray}
&&\sum_{r,s=0}^{g-1}\p(B_{j,\ell,r,s})\nonumber\\
&=&\sum_{r=0}^{g-1}{{\lfloor j/2\rfloor}\choose{r}}p^{2r}(1-p^2)^{\lfloor j/2\rfloor-r}\sum_{s=0}^{g-1}{{\lfloor \ell/2\rfloor-D}\choose{s}}p^{2s}(1-p^2)^{\lfloor \ell/2\rfloor-D-s}(1-p)^D\nonumber\\
&\le&\sum_{r=0}^{g-1}{{\lfloor j/2\rfloor}\choose{r}}p^{2r}(1-p^2)^{\lfloor j/2\rfloor-r}\sum_{s=0}^{g-1}{{\lfloor \ell/2\rfloor}\choose{s}}p^{2s}(1-p^2)^{\lfloor \ell/2\rfloor-s}\lr\frac{1-p}{1-p^2}\rr^D\nonumber\\
&\le&\p(I_j=1)\p(I_\ell=1).
\end{eqnarray}
In the above array, we have denoted by $D$ the number of $\ell$-sumsets that have an overlap with the chosen $j$-sumsets, where $0\le D\le 2r$.  
We must not choose the second of the two integers that give the $\ell$ sumset; this explains the $(1-p)^D$ term.  

We next turn to $\sum_{r,s=0}^{g-1}\p(B^C_{j,\ell,r,s})$, and see that
\begin{eqnarray}
\sum_{r,s=0}^{g-1}\p(B^C_{j,\ell,r,s})&=&
\sum_{r=0}^{g-1}{{\lfloor j/2\rfloor}\choose{r}}p^{2r}(1-p^2)^{\lfloor j/2\rfloor-r}\sum_{s=0}^{g-1}
\sum_{t=1}^{D\wedge s}{D\choose t}p^t(1-p)^{D-t}\times\nonumber\\&&{{\lfloor \ell/2\rfloor-D}\choose{s-t}}p^{2s-2t}(1-p^2)^{\lfloor \ell/2\rfloor-s-D+t}\nonumber\\
&=&\sum_{r=0}^{g-1}{{\lfloor j/2\rfloor}\choose{r}}p^{2r}(1-p^2)^{\lfloor j/2\rfloor-r}\sum_{s=0}^{g-1}p^{2s}(1-p^2)^{\lfloor \ell/2\rfloor-s}\times\nonumber\\
&&\sum_t{D\choose t}{{\lfloor\ell/2\rfloor-D}\choose{s-t}}\frac{(1-p)^{D-t}(1-p^2)^{t-D}}{p^{t}}\nonumber\\
&\le&\sum_{r=0}^{g-1}{{\lfloor j/2\rfloor}\choose{r}}p^{2r}(1-p^2)^{\lfloor j/2\rfloor-r}\sum_{s=0}^{g-1}p^{2s}(1-p^2)^{\lfloor \ell/2\rfloor-s}\times\nonumber\\
&&\sum_{t=1}^{D\wedge s}\frac{{D\choose t}{{\lfloor\ell/2\rfloor-D}\choose{s-t}}}{p^{t}}.
\end{eqnarray}
Consider the summand $\ph(t)$ in the third sum in (11).  We see, on ignoring $\Theta(1)$ terms, that
\[\frac{\ph(t+1)}{\ph(t)}=\frac{1}{p}\frac{1}{\lfloor \ell/2\rfloor-D-s+(t+1)}=\Theta\lr\frac{1}{np}\rr(\to0),\]
so that $\ph$ is decreasing in $t$.  It follows that
\[\sum_{t=1}^{D\wedge s}\frac{{D\choose t}{{\lfloor\ell/2\rfloor-D}\choose{s-t}}}{p^{t}}\le
\frac{D^2\cdot{{\lfloor\ell/2\rfloor-D}\choose{s-1}}}{p}\le\frac{D^2{{\lfloor\ell/2\rfloor}\choose{s-1}}}{p}=O\lr\frac{{{\lfloor\ell/2\rfloor}\choose{s}}}{np}\rr,\]
and thus, by (11)
$$\sum_{r,s=0}^{g-1}\p(B^C_{j,\ell,r,s})\le O\lr\frac{\p(I_j=1)\p(I_\ell=1)}{np}\rr.$$
Equation (9) thus yields, for some constant $K$,
\[\p(I_jI_\ell=1)\le\p(I_j=1)\p(I_\ell=1)\lr1+\frac{K}{np}\rr.\] This proves Lemma 2.1.\hfill\end{proof}  Returning to (7), using (8), we see that for another constant $L$,
\begin{eqnarray}
\tv(\cl(X),\P(\l))&\le&L\lr\frac{\a np^2}{2}\rr^{g-1}e^{-\a np^2/2}+\frac{K}{\l np}\sum_j\p(I_j=1)\sum_\ell \p(I_\ell=1)\nonumber\\
&=&L\lr\frac{\a np^2}{2}\rr^{g-1}e^{-\a np^2/2}+\frac{K\l}{np}.
\end{eqnarray}
Thus $X$ may be approximated by a Poisson random variable provided that $np^2\to\infty$ and $\l\ll np$.  The first condition may be seen to hold if, e.g.,
\[p\gg{\sqrt{\frac{\log\log n}{n}}},\]
and the second if the $\l$ given by (6) is (roughly speaking) of order smaller than $np\sim\sqrt{n\log n}$.  We have thus established Theorem 1.9 for a range of $p$'s that spans part of the $\l\to0$ and $\l\to\infty$ regimes; the full theorem, including the delicated behavior at the threshold, follows easily by monotonicity  (e.g., if $\l$ is even larger than $np$ then it is even less likely that $\p(X=0)$, so that this quantity tends to zero as well).


\section{Open Questions} Establishing an analog of Theorem 1.9 for $h\ge3$ would, of course, be of great interest.  Combating the fact that $h$ sums are not disjoint is the main technical hurdle we would need to overcome.  

Our second open problem is related to an original question of Sidon; see, e.g., \cite{o}, and pertains more to Theorem 1.7.  It has been suggested by Kevin O'Bryant.  Sidon's original question was ``How thick can a set $A\subseteq{\mathbb Z}^+$ be if 
\[\sigma(n)=\vert\{(a,b): a+b\in A; a+b=n\}\vert\]
and
\[\delta(n)=\vert\{(a,b): a-b\in A; a-b=n\}\vert\]
satisfy, for each $n$, $\sigma(n)+\delta(n)\le g$."  Note that in this ordered set format, Sidon sets are those for which $\sigma(n)\le 2$ for each $n$.  It is easy to verify that $\sigma(n)\le 2\ {\rm iff}\ \delta(n)\le 1$.  But if $\sigma(n)\le 4$ then it is still possible for $\delta(n)$ to be unbounded.  Sidon's original question has not been the subject of a large-scale investigation.  In our context, however, we might ask for thresholds for the property $\sigma(n)+\delta(n)\le g$.

\section{Acknowledgments} The research of all four authors was supported by NSF Grant 13XXXX


\begin{thebibliography}{99}
\bibitem{as} N.~Alon and J.~Spencer (1992).  \textit{ The Probabilistic Method.} Wiley, New York.
 \bibitem{bh}   A. Barbour and L. Holst (1989).  Some applications of the Stein-Chen method for proving Poisson convergence, \textit{Adv. Appl. Probab.} {\bf 21}, 74--90.
\bibitem{bhj} A.~Barbour, L.~Holst, and S.~Janson (1992).  \textit{ Poisson Approximation.}  Oxford University Press.
\bibitem{et} P.~Erd\H os and P.~Tetali (1990). Representations of integers as the sum of $k$ terms, \textit{Rand. Structures Algorithms} {\bf 1}, 245--261.
\bibitem{gghk} A.~Godbole, T.~Grubb, K.~Han, and B.  Kay (2017+).  Threshold Progressions in a Variety of Covering and Packing Contexts.  Preprint.
\bibitem{gglz} A.~Godbole, S.~Gutekunst, V.~Lyzinski, and Y.~Zhuang (2015).  Logarithmic Representability of Integers as $k$-Sums, \textit{Integers: Electronic Journal of Combinatorial
Number Theory} {\bf 15A}, Paper \#A5.
\bibitem{gjlr} A.~Godbole, S.~Janson, N.~Locantore, and R.~Rapoport (1999).  Random Sidon sequences,  \textit{J. Number Theory} {\bf 75}, 7--22.
\bibitem{gllt} A.~Godbole, C-M.~Lim, V.~Lyzinski, and N.~Triantafillou (2013).  Sharp threshold asymptotics for the emergence of additive bases, \textit{Integers:  Electronic Journal of Combinatorial Number Theory} {\bf 13}, Paper \# A14.
\bibitem{o} K.~O'Bryant (2004).   A Complete Annotated Bibliography of Work Related to Sidon Sequences, \textit{Electr. J. Comb.}, 	Dynamic Survey \#DS11.
\end{thebibliography}
\end{document}